\documentclass[12pt]{article}
\usepackage{times}
\usepackage{graphicx}
\usepackage{fullpage}
\usepackage{latexsym}
\usepackage{amsmath}
\usepackage{amssymb}
\usepackage{bm}
\usepackage{wrapfig}
\usepackage[english]{babel}
\usepackage{amsthm}
\usepackage{float}
\usepackage{subfigure}
\usepackage{epstopdf}
\usepackage{enumitem}
\numberwithin{equation}{section}

\newtheorem{theorem}{Theorem}
\numberwithin{theorem}{section}

\theoremstyle{definition}
\newtheorem{definition}{Definition}[section]
\newtheorem{corollary}{Corollary}[section]
\newtheorem{proposition}{Proposition}[section]

\theoremstyle{remark}
\newtheorem{example}{Example}[section]

\title{Smarandache Ruled Surfaces According to Bishop Frame in $E^3$}
\author{S\"{u}leyman \uppercase{\c{S}enyurt}$^1$, \quad 	Davut CANLI$^{2*}$, \quad Kebire Hilal AYVACI$^3$ \\ \\
$^{1,2,3}$	Ordu University / Faculty of Arts and Science\\
	Department of Mathematics\\
	$^{*}: \textit{corresponding author}$
}
\date{}

\begin{document}

%\shorthandoff{=}

\maketitle

\begin{abstract}
This paper introduces new ruled surfaces according to Bishop frame by referring the main idea of Smarandache geometry.  The fundamental forms and the corresponding curvatures  are provided to put forth some characteristics of each surface. Finally, an example is given to illustrate the constructed surfaces.
\end{abstract}


\section{Introduction}
\label{intro}
Ruled surfaces are special kind of surfaces that are easy to handle and have the potential use of related fields such as engineering, computational constructions, architectural structures, computer graphics, works of art, geometric designs, textile, automobile industry, etc. The basic theory related to ruled surfaces can be found in many differential geometry textbooks such as \cite{docarmo, gray, haci, struik}. Generalization of ruled surfaces was introduced by Juza in the 1960s \cite{5}. In addition, some characteristic properties of the ruled surface with Frenet frame of a non-cylindrical ruled surface were investigated in 2020 by Ouarab and Chahdi in \cite{8}. Apart from Frenet frame, in \cite{masals} and in \cite{tuncer}, Tun\c{c}er, (2015) and Masal and Azak, (2019), separately, studied some characteristics of the ruled surfaces according to Bishop frame (introduced by Bishop, 1975 in \cite{bishop}), whereas Ouarab et al. (2018) provided the main properties of ruled surfaces according to alternative frame in \cite{14}.\\
Recently, Ouarab, (2021a) put forth a method to generate new ruled surfaces by taking the advantage of the idea of Smarandache geometry \cite{ali,Melih}. By assigning the base curve as one of the Smarandache curves and taking the generator as the another vector element of Frenet frame, she introduced these ruled surfaces as Smarandache ruled surfaces according to Frenet frame in \cite{Ouarab1}. The same method of generating such ruled surfaces is applied to the Darboux frame by Ouarab, (2021b) in \cite{Ouarab3} and according to the alternative frame by Ouarab, (2021c) in \cite{Ouarab2}. There are other studies coined this practice. For example, in 2018 Yilmaz and \c{S}ahin introduced tangent and normal ruled surfaces with TN- Smarandache curve as a base and examined the geodesic conditions of these surfaces \cite{yilmazs3}. They also worked on geodesics of a binormal surface defined by Smarandache curve in \cite{yilmazs2}.\\
Motivated by these, in this study, we address new Smarandache ruled surfaces this time according to the Bishop frame. Then, we study some characteristics of these ruled surfaces and present an example to illustrate each surface.

\section{Preliminaries}\label{sec:2}
In this section, we recall some basic notions of which we refer through out the paper.\\
Let $\gamma= \gamma (s)$ be a regular \underline{unit speed} curve in $E^3$ and denote $\{T,\;N,\;B\}$ as the Frenet frame and $\{T,\;N_1,\;N_2\}$ as the Bishop frame of $\gamma$.
Then, the corresponding Frenet and Bishop formulae are given as 
\begin{equation}\label{derivatChanges}
\begin{array}{ll}
T'= \kappa N \qquad & \qquad T'= k_1 N_1+k_2 N_2 \\
N' =  - \kappa T + \tau B,\qquad & \qquad  {N_1}' =  - k_1 T,\\ 
B' =  - \tau N \qquad & \qquad {N_2}' =  - k_2 T
\end{array}
\end{equation}
where $'$ stands for the derivative with respect to the arc length parameter $s$. The relations among the components and the curvatures of two frames are given as:
\begin{equation}\label{transitions}
\begin{aligned}
T&=\gamma ' ,\\
N&=cos\theta N_1 +sin \theta N_2,\\
B&=-sin \theta N_1 + cos \theta  N_2,
\end{aligned}
\end{equation}
and
\begin{equation}\label{transCurvats}
	\begin{aligned}
		k_1 &= \kappa cos \theta, \qquad k_2 = \kappa sin \theta,\qquad \kappa = \sqrt{k_1 ^2 +k_2 ^2}\, ,\\
		\theta &=arctan \left(\frac{k_2}{k_1}\right), \qquad \tau=\theta ' .
	\end{aligned}
\end{equation}
On the other hand, a surface is said to be ruled if it is formed with a straight line $r(s)$ that moves along the curve $\gamma(s)$. The parametric representation for a ruled surface is given by the following:
\begin{equation}\label{genRuledSurf}
	\chi(s,v) = \gamma (s) + v X(s),
\end{equation}
where $\gamma(s)$ is the base curve, whereas $X(s)$ is the generator (ruling).
The unit normal vector field of $	\chi=	\chi(s,v)$ is computed as
\begin{equation}\label{normalOfSurf}
{n_{\chi}} = \frac{{{{\chi}_s} \times {{\chi}_v}}}{{\left\| {{{\chi}_s} \times {{\chi}_v}} \right\|}},
\end{equation}
where ${\chi}_s$ and ${\chi}_v$ are the partial derivatives of ${\chi}$ with respect to $s$ and $v$, respectively. The striction curve of the ruled surface $\chi$ is defined to be as
\begin{equation}\label{strict}
	{{\bar{\gamma}}} = \gamma-\frac{\langle \gamma ', X ' \rangle}{ \|X' \|^2}X.
\end{equation}
The normal curvature, geodesic curvature and the geodesic torsion of the surface $\chi(s,v)$ associated by its base curve $\gamma$ is given as:
\begin{equation}\label{curvatsTorsi}
\kappa _n =\langle \gamma'', {n_{\chi}} \rangle, \qquad \kappa_g=\langle {n_{\chi}} \times T, T' \rangle, \qquad \tau _g=\langle {n_{\chi}} \times {n_{\chi}} ', T' \rangle,
\end{equation}
respectively. For these invariants, the following propositions exist:
\begin{proposition}\label{propsIlk} (see \cite{docarmo,gray,haci}.)
	\begin{itemize}
	\item The curve $\gamma$ is said to be an asymptotic line of the surface $\chi$ if the normal curvature $\kappa _n$ vanishes,
	\item The curve $\gamma$ is said to be geodesic curvature if the geodesic curvature $\kappa _g$ vanishes,
	\item The curve $\gamma$ is said to be a principal line of the surface $\chi$ if the geodesic torsion $\tau _g$ vanishes.
	\end{itemize}
\end{proposition}
Moreover, the fundamental forms of first and second are defined by
\begin{equation}\label{fundForms}
	\begin{aligned}
		I &= Ed{s^2} + 2Fdsdv + Gd{v^2}, \\
		II &= Ld{s^2} +2Mdsdv + Nd{v^2},
	\end{aligned}
\end{equation}
where the corresponding coefficients are calculated with following 
\begin{equation}\label{coeffs}
	\begin{aligned}
		E &= \left\langle {{\chi _s},{\chi _s}} \right\rangle , \qquad {\rm{ F}} = \left\langle {{\chi _s},{\chi _v}} \right\rangle , \qquad {\rm{ G}} = \left\langle {{\chi _v},{\chi _v}} \right\rangle,\\
		L &= \left\langle {{\chi _{ss}},n} \right\rangle ,\qquad {\rm{ M}} = \left\langle {{\chi _{sv}},n} \right\rangle ,\qquad {\rm{ N}} = \left\langle {{\chi _{vv}},n} \right\rangle.
	\end{aligned}
\end{equation}
Regarding the given coefficients, the Gaussian, $K$ and the mean, $H$ curvatures are
\begin{equation}\label{curvats}
K=\frac{LN-M^2}{EG-F^2},\qquad H=\frac{EN-2MF+LG}{2(EG-F^2)}.
\end{equation} 
In relation to the Gaussian and mean curvatures, we have the following propositions:
\begin{proposition}\label{props} (see \cite{docarmo,gray,haci})$ $
\begin{itemize}
\item 	A surface is said to be developable and to have parabolic points if the Gaussian curvature vanishes,
\item A surface is said to have hyperbolic (resp. elliptic) points if it has the negative (resp. positive) Gaussian curvature,
\item	A surface is said to be minimal if the mean curvature vanishes.
\end{itemize}
\end{proposition}
Lastly, for the purposes of current paper, we refer to the following theorem:
\begin{theorem}\label{teoAl} \big(see, \cite{bukcu, masals} \big)
	If the unit vector $N_1$ has a constant angle with the fixed unit vector, then the curve $\gamma$ is said to be a slant helix. In addition, $\gamma$ is a slant helix iff $\; \; \displaystyle \frac{k_1}{k_2} = constant$.
\end{theorem}
\section{Smarandache ruled surfaces according to Bishop frame in $\mathbf{E^3}$}
\begin{definition}
	Let $\gamma=\gamma(s)$ be a unit speed curve and denote $\{T,\;N_1,\;N_2\}$ as the Bishop frame of $\gamma$. We define and consider three ruled surfaces parameterized as:
\begin{equation}\label{defins}
\begin{aligned}
{^{T{N_1}} _{N_2}} \chi(s,v) &=\frac{T+N_1}{\sqrt{2}}+v N_2,\\
{^{T{N_2}} _{N_1}} \chi(s,v) &=\frac{T+N_2}{\sqrt{2}}+v N_1,\\
{^{{N_1} {N_2}} _T} \chi(s,v) &=\frac{N_1+N_2}{\sqrt{2}}+v T.
\end{aligned}
\end{equation}
By referring the study of Ouarab, (2021a), we name these ruled surfaces as $T N_1$, $T N_2$ and $N_1 N_2$ Smarandache ruled surfaces according to Bishop frame, respectively.
\end{definition}
By considering the relations, (\ref{derivatChanges}) and (\ref{transitions}), the first and second partial derivatives of the surfaces are given in respective order as follows:
\begin{equation*}
{^{T{N_1}} _{N_2}}\chi \rightarrow\left\{
\begin{aligned}
	{{^{T{N_1}} _{N_2}}\chi(s,v)}_s&=-\left( \frac{\sqrt {2}}{2}k_{{1}}+vk_{{2}} \right)T+\frac{\sqrt {2}}{2}k_{{1}} N_1+ \frac{\sqrt {2}}{2}k_{{2}} N_2,\\
	{{^{T{N_1}} _{N_2}}\chi(s,v)}_{ss}&=\left(-\frac{\sqrt {2}}{2} \left(\kappa ^2 + k_{{1}}' \right) -v k_{{2}}'\right) T +\left(  \frac{\sqrt {2}}{2} k_{{1}}' - \left( \frac{\sqrt {2}}{2 } k_{{1}} +v k_{{2}} \right) k_{{1}}\right) N_1\\
	&\quad +  \left( \frac{\sqrt {2}}{2} k_{{2}}' - \left( \frac{\sqrt {2}}{2} k_{{1}}+v\,k_{{2}} \right) k_{{2}} \right) N_2,\\
	{{^{T{N_1}} _{N_2}}\chi(s,v)}_{v}&=N_2,\qquad {{^{T{N_1}} _{N_2}} \chi(s,v)}_{sv}=-k_2 T,\qquad {{^{T{N_1}} _{N_2}}\chi(s,v)}_{vv}=0,
	\end{aligned}
\right.
\end{equation*}
\begin{equation*}
{^{T{N_2}} _{N_1}}\chi \rightarrow \left\{
\begin{aligned}
{{^{T{N_2}} _{N_1}} \chi(s,v)}_s&=-\left( \frac{\sqrt{2}}{2}k_2 + vk_1 \right)T+ \frac{\sqrt{2}}{2} k_1 N_1+\frac{\sqrt{2}}{2} k_2  N_2,\\
{{^{T{N_2}} _{N_1}} \chi(s,v)}_{ss}&=\left(-\frac{\sqrt {2}}{2} \left(\kappa ^2 + k_{{2}}' \right) -v k_{{1}}'\right) T +\left(  \frac{\sqrt {2}}{2} k_{{1}}' - \left( \frac{\sqrt {2}}{2 } k_{{2}} +v k_{{1}} \right) k_{{1}}\right) N_1\\
&\quad + \left( \frac{\sqrt {2}}{2} k_{{2}}' - \left( \frac{\sqrt {2}}{2} k_{{2}}+v k_{{1}} \right) k_{{2}} \right) N_2,\\
{{^{T{N_2}} _{N_1}} \chi(s,v)}_{v}&=N_1,\qquad {{^{T{N_2}} _{N_1}} \chi(s,v)}_{sv}=-k_1 T ,\qquad {{^{T{N_2}} _{N_1}} \chi(s,v)}_{vv}=0,
\end{aligned}
\right.
\end{equation*}
\begin{equation*}
{^{{N_1}{N_2}} _{T}}\chi \rightarrow \left\{
\begin{aligned}
{{{^{{N_1} {N_2}} _T}} \chi(s,v)}_s&=-\frac{\sqrt{2}}{2}\left( k_1+k_2\right)T+ v k_1 N_1+ v k_2 N_2,\\
{{^{{N_1} {N_2}} _T} \chi(s,v)}_{ss}&=-\left(v\kappa ^2 + \frac{\sqrt{2}}{2} (k_1 ' + k_2 ') \right) T +\left( v k_1 ' - \frac{\sqrt{2}}{2} k_1 (k_1+k_2) \right) N_1\\
&\quad + \left( vk_2 ' - \frac{\sqrt{2}}{2} k_2 (k_1 + k_2) \right) N_2,\\
{{^{{N_1} {N_2}} _T} \chi(s,v)}_{v}&=T,\qquad {{^{{N_1} {N_2}} _T} \chi(s,v)}_{sv}=k_1 N_1 + k_2 N_2,\qquad {{^{{N_1} {N_2}} _T} \chi(s,v)}_{vv}=0.
\end{aligned}
\right.
\end{equation*}
By using (\ref{normalOfSurf}), the unit normal vector fields of each surface is given as 
\begin{equation}\label{normals}
\begin{aligned}
	{^{T{N_1}} _{N_2}} n_{\chi} &=\frac{\sqrt{2} k_1 T + (\sqrt{2} k_1 + 2vk_2)N_1}{2\sqrt{ {k_1}^2 + v^2 {k_2}^2 + v k_1 k_2 \sqrt{2}}},\\ \\
	{^{T{N_2}} _{N_1}} n_{\chi}&=-\frac{\sqrt{2} k_2 T + (\sqrt{2} k_2 + 2vk_1)N_2}{2\sqrt{ {k_2}^2 + v^2 {k_1}^2 + v k_1 k_2 \sqrt{2}}} \\ \\
	{^{{N_1} {N_2}} _T} n_{\chi} &= \frac{ k_2 N_1- k_1 N_2}{ \kappa }.
\end{aligned}	
\end{equation}\\
Thus, by taking the relations (\ref{coeffs}), (\ref{curvats}) and (\ref{normals}) into account, we can provide the upcoming theorems without proofs for Gaussian and mean curvatures:
\begin{theorem} \label{teo1}
	The Gaussian and mean curvature of the ruled surface, ${^{T{N_1}} _{N_2}} \chi $ are
	\begin{align*}
	{{^{T{N_1}} _{N_2}} \chi }_{K} &=-\frac{1}{2} \left( \frac { k_{{1}} k_{{2}}}{k_{{1}}^{2} + {v}^{2} k_{{2}}^{2}+  v k_{{1}}k_{{2}}\sqrt {2}}\right)^{2}, \\
	{{^{T{N_1}} _{N_2}} \chi }_H&=\frac{k_{{1}} k_{{2}}^{2} \left( 1-2\,{v}^{2}\right)  + vk_{{2}} \left( k_{{1}}' \sqrt {2} -2 k_{{1}}^{2} \sqrt {2} \right) - vk_{{1}} k_{{2}}'  \sqrt {2}-2\, k_{{1}}^{3}}{4\, \left( k_{{1}}^{2}+{v}^{2} k_{{2}}^{2}+ v k_{{1}} k_{{2}} \sqrt {2} \right) ^{\frac{3}{2}}}.
	\end{align*}
\end{theorem}
As a result of the theorem \ref{teo1} and the given proposition \ref{props}, we have the following corollaries:
\begin{corollary}\label{cor1} $ $
\begin{enumerate}[label=(\roman*)]
	\item The $T N_1$ Smarandache ruled surface has hyperbolic points,
	\item The $T N_1$ Smarandache ruled surface is developable iff the curve $\gamma$ is a planar curve,
	\item The $T N_1$ Smarandache ruled surface is either minimal or constant-mean-curvature (CMC) surface iff the curve $\gamma$ is a planar curve.
\end{enumerate}
\end{corollary}
\begin{proof}$ $
\begin{enumerate}[label=(\roman*)]
	\item The proof is clear.
	\item Let as assume the $T N_1 $ Smarandache ruled surface is developable then ${{^{T{N_1}} _{N_2}} \chi }_{K}=0$, consequently $k_1 k_2 =0$. If $k_1=0$, then from equations (\ref{transCurvats}),	$ \displaystyle \theta = \frac{\pi}{2}k, \; k\in Z $, similarly if $k_2=0$, then $\theta = \pi k, \; k\in Z $. Since $\tau = \theta'$, we get $\tau=0$ for either case. This corresponds to that $\gamma$ is a planar curve.
	\item This is equivalent to the arguments of proof (ii), that is if $k_1=0$ then, ${{^{T{N_1}} _{N_2}} \chi }_{H}=0$, accordingly ${{^{T{N_1}} _{N_2}} \chi }$ is minimal, however if $k_2 =0$, then $\displaystyle {{^{T{N_1}} _{N_2}} \chi }_{H}=-\frac{1}{2}$ which means ${{^{T{N_1}} _{N_2}} \chi }$ is a (CMC) surface.
\end{enumerate}
\end{proof}
\begin{theorem}
The striction curve of the $T N_1$ Smarandache ruled surface is given as
\[ {^{T{N_1}} _{N_2}} \varsigma=\frac{T+N_1-k_1 k_2 N_2}{\sqrt{2}} .\]
\end{theorem}
\begin{proof}
	The derivatives of both the base curve and the generator of the $T N_1 $ Smarandache ruled surface given in \ref{defins}, are
	\begin{align*}
	\left(\frac{T + N_1}{\sqrt{2}}\right)'&=\frac{1}{\sqrt{2}}\left(-k_1 T +k_1 N_1 + k_2 N_2\right),,\\
	{N_2} ' &= -k_2 T.
	\end{align*}
	The inner product of the above two relations is that $\displaystyle \frac{k_1 k_2}{\sqrt{2}}$. By considering the equation (\ref{strict}), we complete the proof.
\end{proof}
\begin{theorem}
The normal curvature, geodesic curvature and the geodesic torsion of the $T N_1 $ Smarandache ruled surface is
\begin{align}\label{tn1Curvats}
{^{T{N_1}} _{N_2}} {\kappa_n} &=\frac{{{k_1}\left( {{k_2}' + {k_1}^2 + {k_2}{k_1} + \left( {\sqrt 2 v + 1} \right){k_2}^2} \right) - {k_1}'\left( {{k_2}\sqrt 2 v + {k_1}} \right)}}{{2\sqrt {{v^2}{k_2}^2 + {k_1}{k_2}v\sqrt 2  + {k_1}^2} }},\nonumber \\
{^{T{N_1}} _{N_2}} {\kappa_g} &=\frac{{{k_1}^2\sqrt 2 \left( {{k_2}^2v\sqrt 2  - {k_2}'} \right) + {k_1}\sqrt 2 \left( {2{k_2}^3 + {k_1}'{k_2}} \right) + {k_2}\left( {4{k_2}^3v + {k_1}^3\sqrt 2 } \right)}}{{2\left( {2{k_2}^2 + {k_1}^2} \right)\sqrt {{v^2}{k_2}^2 + {k_1}{k_2}v\sqrt 2  + {k_1}^2} }},\\
{^{T{N_1}} _{N_2}} {\tau_g} &=\frac{{{{\rm{\eta }}_1}{{\rm{\lambda }}_3}\left( {\sqrt 2 {k_1} + 2{k_2}v} \right) - {{\rm{\eta }}_3}{{\rm{\lambda }}_1}\left( {{k_1}\sqrt 2  + 2{k_2}v} \right) + {k_1}\sqrt 2 \left( {{{\rm{\eta }}_3}{{\rm{\lambda }}_2} - {{\rm{\lambda }}_3}{{\rm{\eta }}_2}} \right)}}{{2\sqrt {{v^2}{k_2}^2 + {k_1}{k_2}v\sqrt 2  + {k_1}^2} }}.\nonumber
	\end{align}
respectively.
\end{theorem}
\begin{proof}
By referring the relations in \ref{transitions}, the tangent vector and the derivative of the tangent vector of $T N_1$ Smarandache curve are given as
\begin{equation}
\begin{aligned}
T_{T{N_1}}=&\frac{-k_1 T + k_1 N_1 + k_2 N_2}{\sqrt{{k_2}^2 + 2 {k_1}^2}},\\
{T_{T{N_1}}}' =& {{\rm{\eta }}_1}T + {{\rm{\eta }}_2} N_1 +{{\rm{\eta }}_3} N_2
\end{aligned}
\end{equation}
where 
\[\left[ \begin{array}{l}
{{\rm{\eta }}_1}\\
{{\rm{\eta }}_2}\\
{{\rm{\eta }}_3}
\end{array} \right] = \frac{1}{{{{\left( {{k_2}^2 + 2{k_1}^2} \right)}^{\frac{3}{2}}}}}\left[ \begin{array}{l}
- {k_2}^2\left( {{k_2}^2 + 3{k_1}^2} \right) - 2{k_1}^4 + {k_2}\left( {{k_1}{k_2}' - {k_2}{k_1}'} \right)\\
- {k_1}^2\left( {{k_2}^2 + 2{k_1}^2} \right) + {k_2}\left( {{k_2}{k_1}' - {k_1}{k_2}'} \right)\\
{k_1}\left( { - {k_2}^3 - 2\left( {{k_2}{k_1}^2 - {k_1}{k_2}' + {k_2}{k_1}'} \right)} \right)
\end{array} \right].\]
On the other hand the second order derivative of $T N_1$ Smarandache curve is
\[\left(\frac{T + N_1}{\sqrt{2}}\right)''=\frac{{ - \left( {{k_1}^2 + {k_2}^2 + \;{\mkern 1mu} {k_1}'} \right)T + \left( {{k_1}' - {k_1}^2} \right){N_1} + \left( {{\mkern 1mu} {k_2}' - {k_1}{k_2}} \right){N_2}} }{\sqrt{2}}.\]
Lastly, the derivative of the normal vector field of the $T N_1$ Smarandache ruled surface is given as follows:
\[\left({^{T{N_1}} _{N_2}} n_{\chi} \right)'= \lambda_1 T+\lambda_2 N_1 + \lambda_3 N_2,\]
where
\[\left[ {\begin{array}{*{20}{c}}
	{{{\rm{\lambda }}_1}}\\ \\
	{{{\rm{\lambda }}_2}}\\ \\
	{{{\rm{\lambda }}_3}}
	\end{array}} \right] = \frac{1}{{2{{\left( {{v^2}{k_2}^2 + {k_1}{k_2}v\sqrt 2  + {k_1}^2} \right)}^{\frac{3}{2}}}}}\left[ \begin{array}{l}
{k_1}'{k_2}v\left( {{k_2}v\sqrt 2  + {k_1}} \right) - {k_1}{k_2}'v\left( {{k_2}\sqrt 2 v + {k_1}} \right)...\\
\,\,\,\,\,\,\,\,\,\, - {k_1}{k_2}^2v\left( {3{k_1}v\sqrt 2  + 2{k_2}{v^2}} \right) - {k_1}^3\left( {{k_1}\sqrt 2  + 4{k_2}v} \right)\\ \\
{k_1}\left( {{k_1}{k_2}^2\sqrt 2 {v^2} + 2{k_2}{k_1}^2v - {k_2}{k_1}'v + {k_1}^3\sqrt 2  + {k_1}{k_2}'v} \right)\\ \\
{k_1}{k_2}\sqrt 2 \left( {{v^2}{k_2}^2 + {k_1}{k_2}v\sqrt 2  + {k_1}^2} \right)
\end{array} \right].\]
Upon substituting these relations into \ref{curvatsTorsi}, the proof is completed. According to this theorem, we provide the following two corollaries without the need for proof.
\end{proof}
\begin{corollary}$ $
\begin{enumerate}[label=(\roman*)]
	\item The $T N_1$ Smarandache curve is asymptotic on $T N_1$ Smarandache ruled surface if $k_1=0$ that is 
	$ \displaystyle \theta = \frac{\pi}{2}k, \; k\in Z $.
	\item The $T N_1$ Smarandache curve is geodesic on $T N_1$ Smarandache ruled surface if $k_2=0$ that is
	$ \displaystyle \theta = \pi k, \; k\in Z $.
\end{enumerate}
\end{corollary}
%\begin{proof}$ $
%\begin{enumerate}[label=(\roman*)]
%	\item The normal curvature of the $T N_1$ Smarandache ruled surface given in \ref{tn1Curvats} vanishes if $k_1 =0.$ If $k_1=0$, then by the relation \ref{transCurvats}, $ \displaystyle \theta = \frac{\pi}{2}k, \; k\in Z $. Since $\tau = \theta'$ namely $\tau=0$, then $\gamma$ is a planar curve.
%\item If $\gamma$ is a slant helix, then by the Theorem \ref{teoAl}, $k_1 k_2'-k_1' k_2=0$ which clearly vanishes both the geodesic curvature and torsion given in \ref{tn1Curvats}.
%\end{enumerate}
%\end{proof}
\begin{theorem}\label{teo2}
	The Gaussian and mean curvature of the ruled surface, ${^{T{N_2}} _{N_1}}  \chi $ are
	\begin{align*}
	{{^{T{N_2}} _{N_1}}  \chi }_{K} &=-\frac{1}{2} \left( \frac { k_{{1}} k_{{2}}}{k_{{2}}^{2} + {v}^{2} k_{{1}}^{2}+  v k_{{1}}k_{{2}}\sqrt {2}}\right)^{2}, \\
	{{^{T{N_2}} _{N_1}}  \chi }_H&=\frac{k_{{1}}^2 k_{{2}} \left( 1-2\,{v}^{2}\right)  + vk_{{1}} \left( k_{{2}}' \sqrt {2} -2 k_{{2}}^{2} \sqrt {2} \right) - vk_{{1}}' k_{{2}}  \sqrt {2}-2 k_{{2}}^{3}}{4\, \left( k_{{2}}^{2}+{v}^{2} k_{{1}}^{2}+ v k_{{1}} k_{{2}} \sqrt {2} \right) ^{\frac{3}{2}}}.
	\end{align*}
\end{theorem}
By the theorem \ref{teo2} and proposition \ref{props}, we have the similar corollaries as following:
\begin{corollary}\label{cor2} $ $
	\begin{enumerate}[label=(\roman*)]
		\item The $T N_2$ Smarandache ruled surface has hyperbolic points,
		\item The $T N_2$ Smarandache ruled surface is developable iff the curve $\gamma$ is a planar curve,
		\item The $T N_2$ Smarandache ruled surface is either minimal or constant-mean-curvature (CMC) surface iff the curve $\gamma$ is a planar curve.
	\end{enumerate}
\end{corollary}
	\begin{proof}$ $
The proofs of each corollary is as same as the proofs of Corollary \ref{cor1}.
	\end{proof}
\begin{theorem}
	The striction curve of the $T N_2$ Smarandache ruled surface is given as
	\[{^{T{N_2}} _{N_1}} \varsigma=\frac{T-k_1 k_2 N_1+N_2}{\sqrt{2}} .\]
\end{theorem}
\begin{proof}
The derivative of both the base curve and the generator of $T N_2 $ Smarandache ruled surface defined in \ref{defins}, are
\begin{align*}
\left(\frac{T + N_2}{\sqrt{2}}\right)'&=\frac{\sqrt{2}}{2}\left(-k_2 T +k_1 N_1 + k_2 N_2\right),,\\
{N_1} ' &= -k_1 T.
\end{align*}
The inner product of the left hand sides of these relations results $\displaystyle \frac{k_1 k_2}{\sqrt{2}}$. By considering the equation (\ref{strict}), the proof is complete.
\end{proof}

\begin{theorem}
	The normal curvature, geodesic curvature and the geodesic torsion of the $T N_2 $ Smarandache ruled surface is
	\begin{align}\label{tn2Curvats}
{^{T{N_2}} _{N_1}} {\kappa_n} &=\frac{{{k_1}v\sqrt 2 \left( { {k_2}^2 - {k_2}'} \right) + {k_2}\left( {{k_1}^2 + 2{k_2}^2} \right)}}{{2\sqrt {{v^2}{k_1}^2 + {k_2}{k_1}v\sqrt 2  + {k_2}^2} }},\nonumber \\
{^{T{N_2}} _{N_1}} {\kappa_g} &=\frac{{2\left( {{k_1}'{k_2} - {k_1}{k_2}'} \right)\left( {\sqrt 2 {k_2} + {k_1}v} \right) - {k_1}{k_2}\sqrt 2 \left( {2{k_2}^2 + {k_1}^2} \right) - 2{k_1}^2v\left( {{k_1}^2 + 4{k_2}^2} \right)}}{{2\left( {2{k_2}^2 + {k_1}^2} \right)\sqrt {{k_1}{k_2}v\sqrt 2  + {v^2}{k_1}^2 + {k_2}^2} }},\\
{^{T{N_2}} _{N_1}} {\tau_g} &=\frac{{\left( {{{\rm{\alpha }}_2}{{\rm{\omega }}_1} - {{\rm{\alpha }}_1}{{\rm{\omega }}_2}} \right)\left( {{k_2}\sqrt 2  + 2{k_1}v} \right) + {k_2}\sqrt 2 \left( {{{\rm{\alpha }}_3}{{\rm{\omega }}_2} - {{\rm{\alpha }}_2}{{\rm{\omega }}_3}} \right)}}{{2\sqrt {{k_1}{k_2}v\sqrt 2  + {v^2}{k_1}^2 + {k_2}^2} }},\nonumber
	\end{align}
	respectively.
\end{theorem}
\begin{proof}
	Again, by referring the relations in \ref{transitions}, the tangent vector and the derivative of the tangent vector of $T N_2$ Smarandache curve are given as
	\begin{equation}
	\begin{aligned}
	T_{T{N_2}}=&\frac{-k_2 T + k_1 N_1 + k_2 N_2}{\sqrt{{k_1}^2 + 2 {k_2}^2}},\\
	{T_{T{N_2}}}' =&\omega_1 T+ \omega_2 N_1 + \omega_3 N_2,\\
	\end{aligned}
	\end{equation}
	where
	\[\left[ {\begin{array}{*{20}{c}}
		{{{\rm{\omega }}_1}}\\
		{{{\rm{\omega }}_2}}\\
		{{{\rm{\omega }}_3}}
		\end{array}} \right] = \frac{1}{{{{\left( {2{k_2}^2 + {k_1}^2} \right)}^{\frac{3}{2}}}}}\left[ \begin{array}{l}
	{k_1}{\mkern 1mu} \left( {{k_1}'{k_2} - {k_2}'{k_1}} \right) - 2{k_2}^4 - 3{k_1}^2{k_2}^2 - {k_1}^4\\
	{k_2}\left( {2\left( {{k_1}'{k_2} - {k_2}'{k_1}} \right) - 2{k_2}^2{k_1} - {k_1}^3} \right)\\
	{k_1}\left( {{k_2}'{k_1} - {k_1}'{k_2}} \right) - 2{k_2}^4 - {k_1}^2{k_2}^2
	\end{array} \right].\]
On the other hand the second order derivative of $T N_2$ Smarandache curve is
\[\left(\frac{T + N_2}{\sqrt{2}}\right)''=\frac{{ - \left( {{k_1}^2 + {k_2}^2 + {k_2}'} \right)T + \left( {{k_1}' - {k_1}{k_2}} \right){N_1} + \left( { {k_2}' - {k_2}^2} \right){N_2}} }{\sqrt{2}}.\]
	Lastly, the derivative of the normal vector field of the $T N_2$ Smarandache ruled surface is given as follows:
	\[\left({^{T{N_2}} _{N_1}} n_{\chi} \right)'=\alpha_1 T+ \alpha_2 N_1 + \alpha_3 N_2,\]
	where
\[\left[ {\begin{array}{*{20}{c}}
	{{{\rm{\alpha }}_1}}\\ \\
	{{{\rm{\alpha }}_2}}\\ \\
	{{{\rm{\alpha }}_3}}
	\end{array}} \right] = \frac{1}{{2{{\left( {{k_1}{k_2}v\sqrt 2  + {v^2}{k_1}^2 + {k_2}^2} \right)}^{\frac{3}{2}}}}}\left[ \begin{array}{l}
\, \left( {\sqrt 2 {k_1}v + {k_2}} \right)\left( {{k_1}'{k_2} - {k_1}{k_2}} \right)v...\\
\,\,\,\,\, + {k_1}^2{k_2}{v^2}\left( {3{k_2}\sqrt 2  + 2{k_1}v} \right) + {k_2}^3\left( {{k_2}\sqrt 2  + 4{k_1}v} \right)\\ \\
-{k_1}{k_2}\sqrt 2 \left( {{k_1}{k_2}v\sqrt 2  + {v^2}{k_1}^2 + {k_2}^2} \right)\\ \\
\,{k_2}\left( { - {k_2}\sqrt 2 \left( {{k_1}^2{v^2} + {k_2}^2} \right) - v\left( {2{k_1}{k_2}^2 - {k_1}{k_2} + {k_1}'{k_2}} \right)} \right)
\end{array} \right]\]
	When the above relations substituted into \ref{curvatsTorsi}, the proof is completed. As a result of this theorem, we provide the following two corollaries without the need for proof.
\end{proof}
\begin{corollary}$ $
	\begin{enumerate}[label=(\roman*)]
		\item The $T N_2$ Smarandache curve is asymptotic on $T N_2$ Smarandache ruled surface if $k_2=0$ that is
		$ \displaystyle \theta = \pi k, \; k\in Z $.
		\item The $T N_2$ Smarandache curve is geodesic on $T N_2$ Smarandache ruled surface if $k_1=0$ that is
		$ \displaystyle \theta = \frac{\pi}{2} k, \; k\in Z $.
	\end{enumerate}
\end{corollary}
%\begin{proof}$ $
%	\begin{enumerate}[label=(\roman*)]
%		\item The normal curvature of the $T N_1$ Smarandache ruled surface given in \ref{tn2Curvats} vanishes if $k_2 =0.$ If $k_2=0$, then by the relation \ref{transCurvats}, . Since $\tau = \theta'$ namely $\tau=0$, then $\gamma$ is a planar curve.
%		\item If $\gamma$ is a slant helix, then by the Theorem \ref{teoAl}, $k_1 k_2'-k_1' k_2=0$ which clearly vanishes both the geodesic curvature and torsion given in \ref{tn2Curvats}.
%	\end{enumerate}
%\end{proof}

\begin{theorem}\label{teo3}
	The Gaussian and mean curvature of the ruled surface, ${{^{{N_1} {N_2}} _T}} \chi $ are
	\begin{align*}
	{{{^{{N_1} {N_2}} _T}} \chi}_{K} &=0 , \\
	{{{^{{N_1} {N_2}} _T}} \chi}_H&=\frac{k_1 ' k_2 - k_1 k_2 '}{2v \kappa^3}.
	\end{align*}
\end{theorem}
By theorem \ref{teo3}, and proposition \ref{props}, the following corollaries are obtained.

\begin{corollary}\label{cor3} $ $
	\begin{enumerate}[label=(\roman*)]
		\item The $N_1 N_2$ Smarandache ruled surface is always developable and so has parabolic points,
		\item The $N_1 N_2$ Smarandache ruled surface is minimal iff the curve $\gamma$ is a slant helix.
	\end{enumerate}
\end{corollary}
\begin{proof}$ $
	\begin{enumerate}[label=(\roman*)]
		\item The proof is clear by the first item of proposition \ref{props}.
		\item Recalling the Theorem \ref{teoAl}, $\gamma$ is a slant helix iff $\displaystyle \left(\frac{k_1}{k_2}\right)'=0$ that is ${{{^{{N_1} {N_2}} _T}} \chi}_H=0$. Therefore ${{{^{{N_1} {N_2}} _T}} \chi}$ is minimal.\\
		Conversely, if the $N_1 N_2$ Smarandache ruled surface is minimal, namely ${{{^{{N_1} {N_2}} _T}} \chi}_H=0$, then $k_1 ' k_2-k_1 k_2' =0$. Thus, $\displaystyle \frac{k_1}{k_2} = constant$ meaning that $\gamma$ is a slant helix.
	\end{enumerate}
\end{proof}

\begin{theorem}
	The striction curve of the $N_1 N_2$ Smarandache ruled surface is given as
	\[ {^{{N_1}{N_2}} _{T}} \varsigma=\frac{N_1+N_2}{\sqrt{2}}.\]
\end{theorem}
\begin{proof}
The derivative of both the base curve and the generator of $N_1 N_2$ Smarandache ruled surface defined in \ref{defins}, are
\begin{align*}
\left(\frac{N_1+N_2}{\sqrt{2}}\right)'&=-\left(\frac{k_1+k_2}{\sqrt{2}}\right)\, T,\\
T ' &= k_1 N_1 + k_2 N_2.
\end{align*}
By considering the equation (\ref{strict}), the inner product of the left hand side for above relations vanishes, which completes the proof.
\end{proof}
\begin{corollary}$ $
The striction curve of the $N_1 N_2$ Smarandache ruled surface is same as the base curve which is $N_1 N_2$ Smarandache curve of $\gamma$.
\end{corollary}
\begin{theorem}
	The normal curvature, geodesic curvature and the geodesic torsion of the $N_1 N_2 $ Smarandache ruled surface is
	\begin{equation}\label{tn3Curvats}
	{^{{N_1}{N_2}} _{T}} {\kappa_n} =0, \qquad	{^{{N_1}{N_2}} _{T}} {\kappa_g} =-\kappa, \qquad	{^{{N_1}{N_2}} _{T}} {\tau_g} =0,
	\end{equation}
	respectively.
\end{theorem}
\begin{proof}
	This time, by recalling both the relations (\ref{transitions}) and (\ref{transCurvats}), the tangent vector and the derivative of the tangent vector of $N_1 N_2$ Smarandache curve are given as
	\begin{equation}
	\begin{aligned}
	T_{{N_1}{N_2}}=&-T,\\
	{T_{{N_1}{N_2}}}' =&-k_1 N_1 - k_2 N_2.\\
	\end{aligned}
	\end{equation}
	Moreover, the second order derivative of $N_1 N_2$ Smarandache curve is
	\[\left(\frac{N_1 + N_2}{\sqrt{2}}\right)''=-\frac{{\left(k_1 ' + k_2 ' \right)T + \left( k_1 ^2+ k_1 k_2 \right){N_1} + \left( k_1 k_2+ k_2 ^2 \right){N_2}} }{\sqrt{2}}.\]
	Lastly, the derivative of the normal vector field of the $T N_2$ Smarandache ruled surface is given as follows:
	\[\left({^{{N_1}{N_2}} _{T}} n_{\chi} \right)'= {\frac{{{k_1}\left( {{k_1}{k_2}' - {k_1}'{k_2}} \right){N_1} + {k_2}\left( {{k_1}{k_2} - {k_1}'{k_2}} \right){N_2}}}{{{\kappa ^3}}}}.\]
	When the above relations substituted into \ref{curvatsTorsi}, the proof is completed.
\end{proof}
\begin{corollary}$ $
	\begin{enumerate}[label=(\roman*)]
		\item The $N_1 N_2$ Smarandache curve is always asymptotic and principal line on $N_1 N_2$ Smarandache ruled surface.
		\item The The geodesic curvature of $N_1 N_2$ Smarandache ruled surface is the negative of the curvature of the initial curve $\gamma$.
	\end{enumerate}
\end{corollary}
%\begin{proof}$ $
%	\begin{enumerate}[label=(\roman*)]
%		\item The normal curvature of the $T N_1$ Smarandache ruled surface given in \ref{tn2Curvats} vanishes if $k_2 =0.$ If $k_2=0$, then by the relation \ref{transCurvats}, $ \displaystyle \theta = \pi k, \; k\in Z $. Since $\tau = \theta'$ namely $\tau=0$, then $\gamma$ is a planar curve.
%		\item If $\gamma$ is a slant helix, then by the Theorem \ref{teoAl}, $k_1 k_2'-k_1' k_2=0$ which clearly vanishes both the geodesic curvature and torsion given in \ref{tn2Curvats}.
%	\end{enumerate}
%\end{proof}

\begin{example}\label{ornek}
Let us consider the standard unit helix curve parameterized as
\[\gamma(s)=\frac{\sqrt{2}}{2}\bigg(cos(s),sin(s),s\bigg), \]
then, the Frenet curvatures of $\gamma$ are $ \displaystyle \kappa=\tau= \frac{\sqrt{2}}{2}$. Since $ \displaystyle \tau=\theta '$, we have $ \displaystyle \theta=\int \tau ds = \frac{s\sqrt{2}}{2}$, then we establish the corresponding curvatures and the vectors of Bishop frame as 
\[k_1=\frac{\sqrt{2}}{2}cos\left(\frac{s\sqrt{2}}{2}\right)\qquad \text{and} \qquad k_2=\frac{\sqrt{2}}{2}sin\left(\frac{s\sqrt{2}}{2}\right),\]
\begin{align*}
T(s)&=\frac{\sqrt{2}}{2}\bigg(-sin(s),cos(s),1\bigg),\\
N_1 (s)&=\left(\begin{array}{l}
\displaystyle -\cos \left( \frac{s \sqrt {2}}{2} \right) \cos
\left( s \right) - \frac{ \sqrt {2}}{2} \sin \left( \frac{s \sqrt {2}}{2} \right) \sin
\left( s \right) ,\\ 
\displaystyle -\cos \left( \frac{s \sqrt {2}}{2} \right) \sin \left( s \right) + \frac{ \sqrt {2}}{2} \sin \left( \frac{s \sqrt {2}}{2} \right) \cos \left( s \right) ,\\ 
\displaystyle \frac{ \sqrt {2}}{2}\sin \left( \frac{s \sqrt {2}}{2} \right),
\end{array}\right),\\
N_2 (s)&=\left(\begin {array}{l}
\displaystyle -\sin \left( \frac{s \sqrt {2}}{2} \right) \cos
\left( s \right) +\frac{ \sqrt {2}}{2}\,\cos \left( \frac{s \sqrt {2}}{2} \right) \sin
\left( s \right),\\
\displaystyle -\sin \left( \frac{s \sqrt {2}}{2} \right) \sin \left( s \right) -\frac{ \sqrt {2}}{2} \,\cos \left( \frac{s \sqrt {2}}{2} \right) \cos \left( s \right) ,\\
\displaystyle \frac{\sqrt {2}}{2} \cos \left(\frac{s \sqrt {2}}{2}\right) \end {array}
\right).\\
\end{align*}
By referring the given definitions in (\ref{defins}) we have  ruled surfaces that are ${^{T N_1} _{N_2} }\chi(s,v) ,$ ${^{T N_2} _{N_1} }\chi(s,v) $ and ${^{N_1 N_2 } _T }\chi(s,v) $ are given by

\begin{align*}
{^{T N_1} _{N_2} } \chi(s,v) &= \left( \begin{array}{l}
\left( { - \frac{{\sin \left( s \right)}}{2} - v\cos \left( s \right)} \right)\sin \left( {\frac{{s\sqrt 2 }}{2}} \right) + \left( { - \frac{{\cos \left( s \right)\sqrt 2 }}{2} + \frac{{v\sin \left( s \right)\sqrt 2 }}{2}} \right)\cos \left( {\frac{{s\sqrt 2 }}{2}} \right) - \frac{{\sin \left( s \right)}}{2},\\
\left( {\frac{{\cos \left( s \right)}}{2} - v\sin \left( s \right)} \right)\sin \left( {\frac{{s\sqrt 2 }}{2}} \right) + \left( { - \frac{{\sin \left( s \right)\sqrt 2 }}{2} - \frac{{v\cos \left( s \right)\sqrt 2 }}{2}} \right)\cos \left( {\frac{{s\sqrt 2 }}{2}} \right) + \frac{{\cos \left( s \right)}}{2},\\
- \frac{1}{2}\sin \left( {\frac{{s\sqrt 2 }}{2}} \right) + \frac{1}{2}v\cos \left( {\frac{{s\sqrt 2 }}{2}} \right)\sqrt 2  + \frac{1}{2} \end{array} \right), \\ \\
{^{T N_2} _{N_1} }\chi(s,v)  &=\left( \begin{array}{l}
\left( { - \frac{{\cos \left( s \right)\sqrt 2 }}{2} - \frac{{v\sin \left( s \right)\sqrt 2 }}{2}} \right)\sin \left( {\frac{{s\sqrt 2 }}{2}} \right) + \left( {\frac{{\sin \left( s \right)}}{2} - v\cos \left( s \right)} \right)\cos \left( {\frac{{s\sqrt 2 }}{2}} \right) - \frac{{\sin \left( s \right)}}{2},\\
\left( { - \frac{{\sin \left( s \right)\sqrt 2 }}{2} + \frac{{v\cos \left( s \right)\sqrt 2 }}{2}} \right)\sin \left( {\frac{{s\sqrt 2 }}{2}} \right) + \left( { - \frac{{\cos \left( s \right)}}{2} - v\sin \left( s \right)} \right)\cos \left( {\frac{{s\sqrt 2 }}{2}} \right) + \frac{{\cos \left( s \right)}}{2},\\
- \frac{{\sqrt 2 }}{2}v\sin \left( {\frac{{s\sqrt 2 }}{2}} \right) + \frac{1}{2}\cos \left( {\frac{{s\sqrt 2 }}{2}} \right) + \frac{1}{2}
\end{array} \right) ,\\ \\
{^{N_1 N_2 } _T }\chi(s,v) &= \left( \begin{array}{l}
\frac{{\sqrt 2 }}{2}\left( {\left( { - \cos \left( s \right) - \frac{{\sin \left( s \right)\sqrt 2 }}{2}} \right)\sin \left( {\frac{{s\sqrt 2 }}{2}} \right) + \left( {\frac{{\sin \left( s \right)\sqrt 2 }}{2} - \cos \left( s \right)} \right)\cos \left( {\frac{{s\sqrt 2 }}{2}} \right) - v\sin \left( s \right)} \right),\\
\frac{{\sqrt 2 }}{2}\left( {\left( { - \sin \left( s \right) + \frac{{\cos \left( s \right)\sqrt 2 }}{2}} \right)\sin \left( {\frac{{s\sqrt 2 }}{2}} \right) + \left( { - \frac{{\cos \left( s \right)\sqrt 2 }}{2} - \sin \left( s \right)} \right)\cos \left( {\frac{{s\sqrt 2 }}{2}} \right) + v\cos \left( s \right)} \right),\\
\frac{1}{2}\left( { - v\sin \left( {\frac{{s\sqrt 2 }}{2}} \right)\sqrt 2  + \cos \left( {\frac{{s\sqrt 2 }}{2}} \right) + 1} \right)
\end{array}\right).
\end{align*}
respectively. For $s\in [-\pi ,\pi]$ and $v\in (-1,1)$, the graphs of these surfaces are provided in Fig. (\ref{f1}).
\begin{figure}[H]
\centering	\hfill
	\subfigure[${^{T N_1} _{N_2} } \chi$]{\includegraphics[width=0.30\textwidth	]{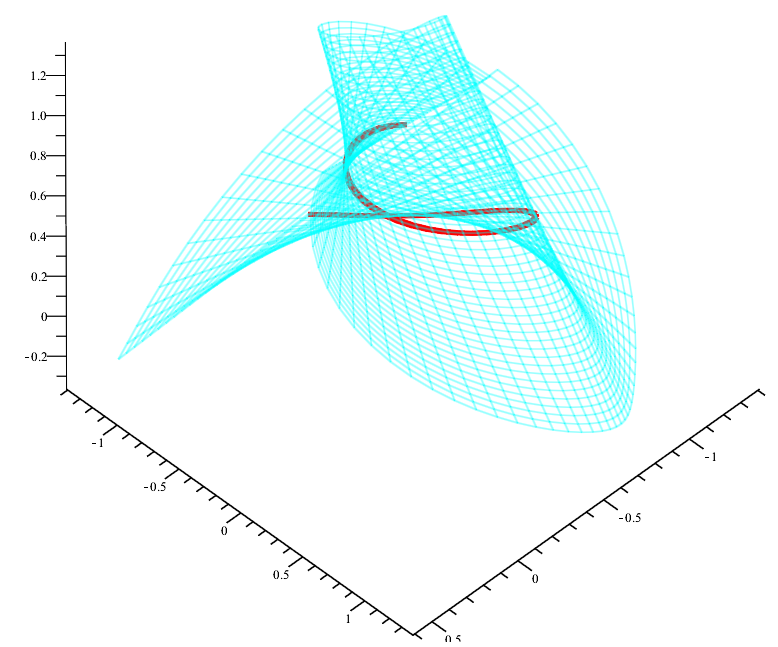}}
	\hfill
	\subfigure[${^{T N_2} _{N_1} } \chi$]{\includegraphics[width=0.30\textwidth	]{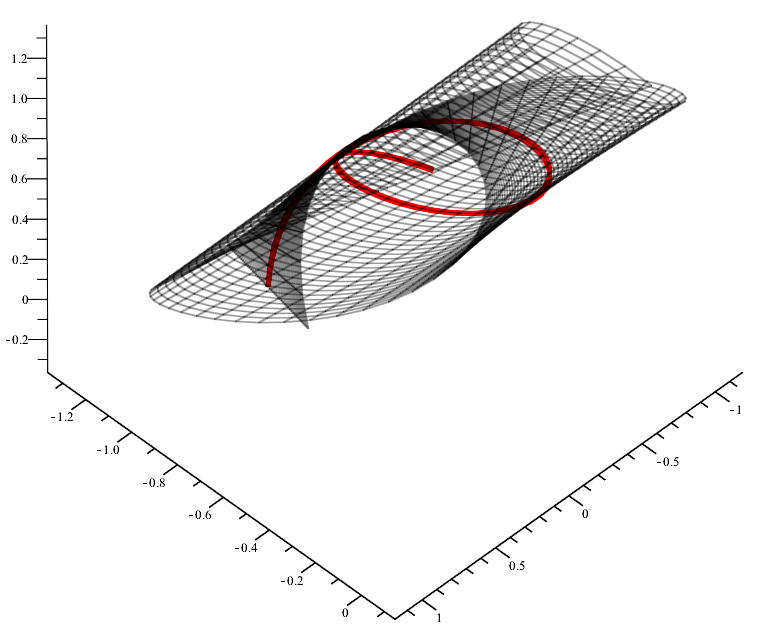}}
	\hfill
	\subfigure[${^{N_1 N_2} _{T} } \chi$]{\includegraphics[width=0.30\textwidth	]{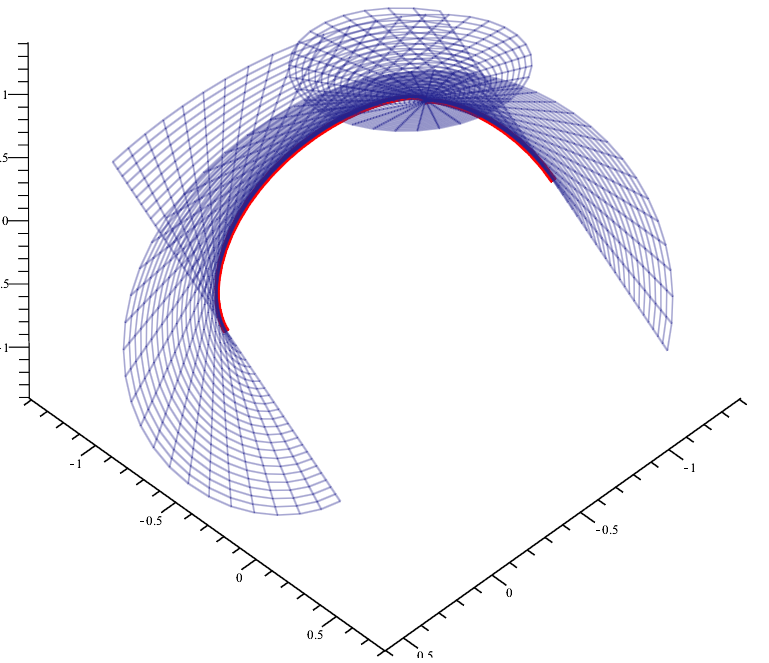}}
	\hfill
	\caption{Smarandache Ruled surfaces from the view of default orientation in Maple} \label{f1}
\end{figure}
\end{example}
The following figures from different aspects are also presented as to view each surface much clearly. The orientations are fixed to the $x$, $y$ and $z$ axis, respectively.
\begin{figure}[H]
	\centering	\hfill
	\subfigure[oriented to $x$ axis]{\includegraphics[width=0.30\textwidth	]{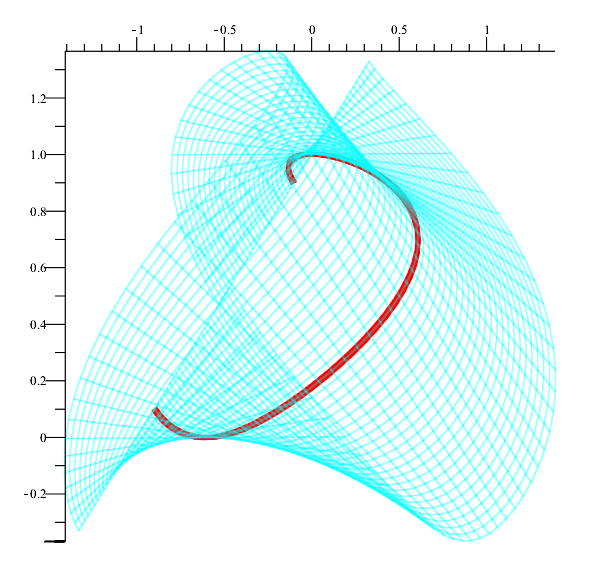}}
	\hfill
	\subfigure[oriented to $y$ axis]{\includegraphics[width=0.30\textwidth	]{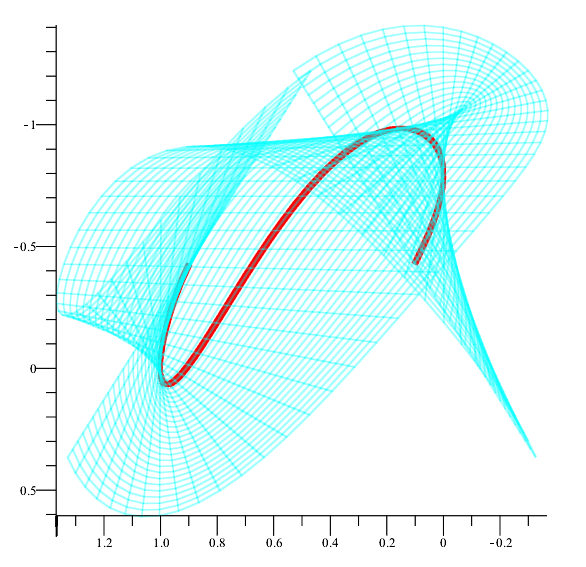}}
	\hfill
	\subfigure[oriented to $z$ axis]{\includegraphics[width=0.30\textwidth	]{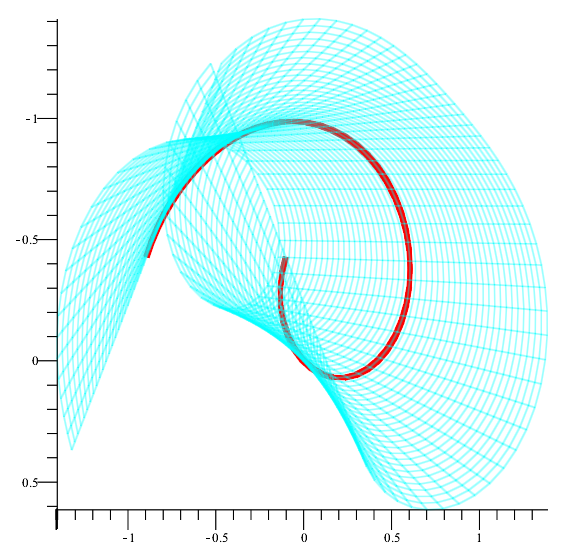}}
	\hfill
	\caption{The ${^{T N_1} _{N_2} } \chi$ Smarandache Ruled surfaces } \label{f2}
\end{figure}
\begin{figure}[H]
	\centering	\hfill
	\subfigure[oriented to $x$ axis]{\includegraphics[width=0.30\textwidth	]{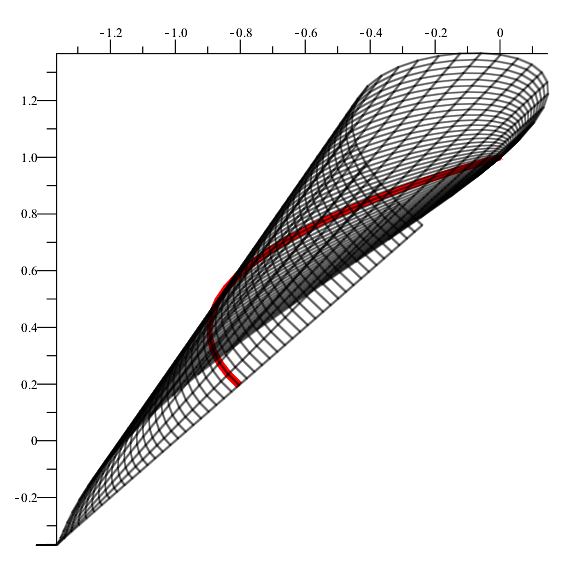}}
	\hfill
	\subfigure[oriented to $y$ axis]{\includegraphics[width=0.30\textwidth	]{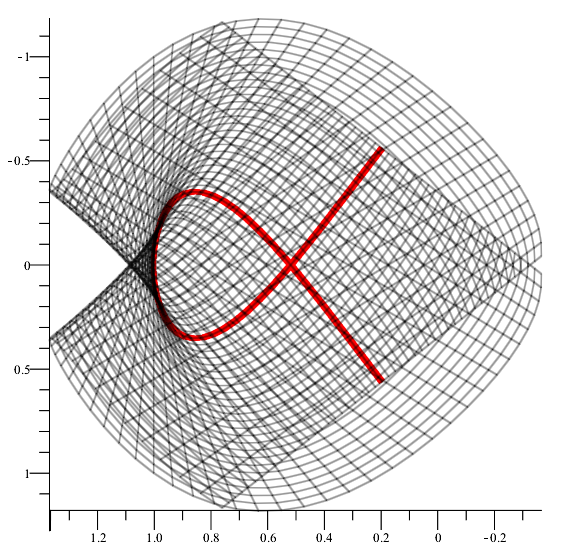}}
	\hfill
	\subfigure[oriented to $z$ axis]{\includegraphics[width=0.30\textwidth	]{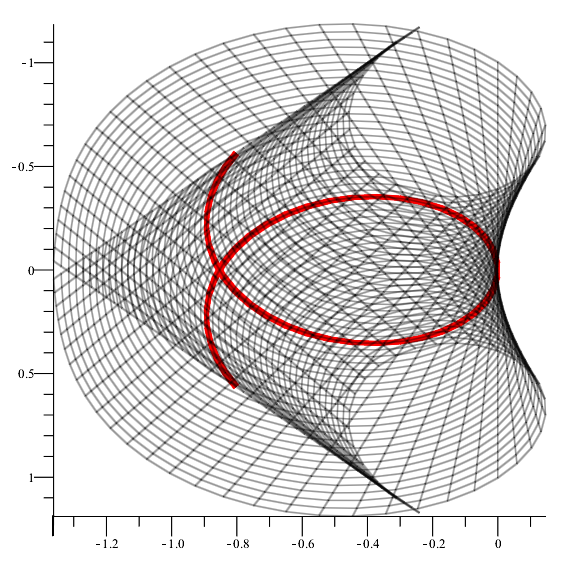}}
	\hfill
	\caption{The ${^{T N_2} _{N_1} } \chi$ Smarandache Ruled surfaces } \label{f3}
\end{figure}
\begin{figure}[H]
	\centering	\hfill
	\subfigure[oriented to $x$ axis]{\includegraphics[width=0.30\textwidth	]{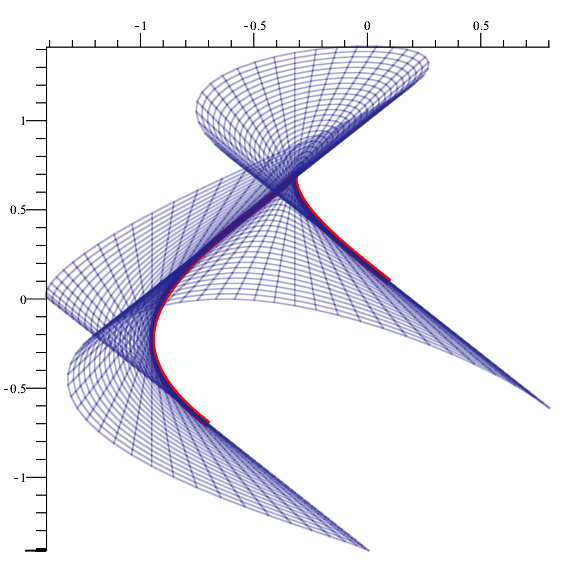}}
	\hfill
	\subfigure[oriented to $y$ axis]{\includegraphics[width=0.30\textwidth	]{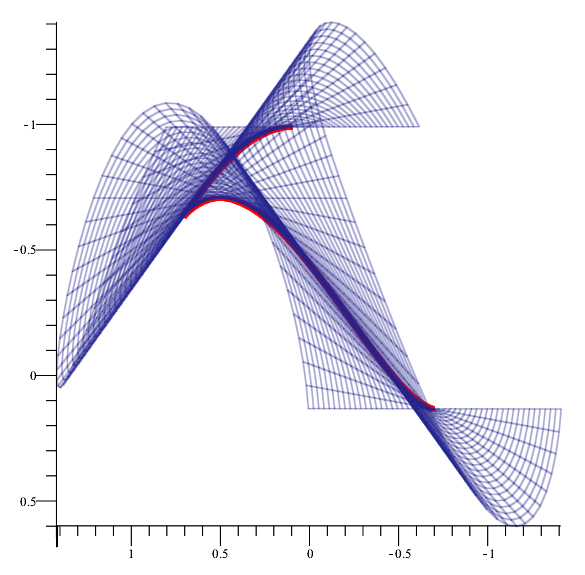}}
	\hfill
	\subfigure[oriented to $z$ axis]{\includegraphics[width=0.30\textwidth	]{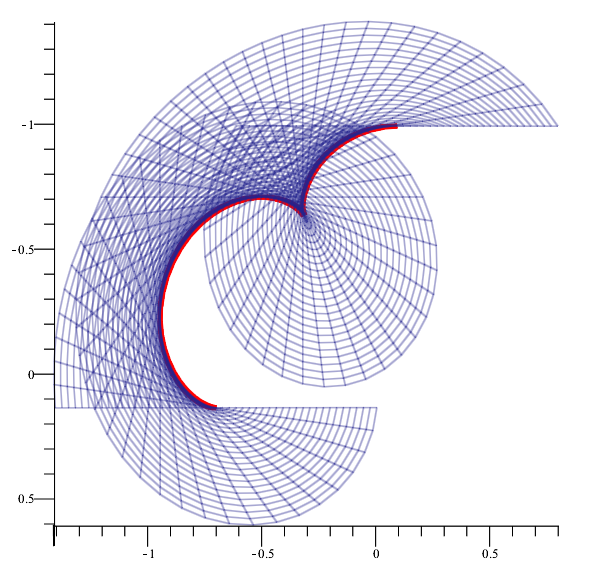}}
	\hfill
	\caption{The ${^{N_1 N_2} _{T} } \chi$ Smarandache Ruled surfaces } \label{f4}
\end{figure}
%%%%%%%%%%%%%%%%%%%%%%%%%%%%%%%%%%%%%%%%%%%%%%%%%%%%%%%%%%%%%%%%%%%%%%%%%%%%%%%%%%%%%%%%%%%%%%%%%%%%%%%%%%%
%%%%%%%%%%%%%%%%%%%%%%%%%%%%%%%%%%%%%%%%%%%%%%%%%%%%%%%%%%%%%%%%%%%%%%%%%%%%%%%%%%%%%%%%%%%%%%%%%%%%%%%%%%%%
\section{Conclusion}
Overall, in the paper, three new ruled surfaces based on Smarandache curves according to Bishop frame have been introduced. The characteristics of each surface and of the curves lying on these surfaces have also been drawn. It is seen that apart from choosing the initial curve as planar, slant helix curves have direct effect on characterizing the ruled surfaces generated by Bishop frame.

\end{document}